\documentclass[11pt]{article}
\usepackage[margin=1in]{geometry}
\usepackage{amsmath,amssymb,amsthm}
\usepackage{booktabs}
\usepackage{float}
\newfloat{algorithm}{tbp}{loa}
\floatname{algorithm}{Algorithm}
\usepackage[labelfont=bf,font=small]{caption}
\usepackage{microtype}
\usepackage[hidelinks]{hyperref}
\usepackage{xcolor}
\usepackage{tikz}
\usetikzlibrary{decorations.pathreplacing}

\newtheorem{proposition}{Proposition}
\newtheorem{lemma}{Lemma}

\newcommand{\Aut}{\operatorname{Aut}}
\newcommand{\dd}{d}

\title{\bf The number of labeled partial orders and topologies on 19 points}
\author{Rafael Ayala\thanks{\texttt{rayala@us.es}}}
\date{}

\begin{document}
\maketitle

\begin{abstract}
We report the exact value of the number of labeled partially ordered sets (equivalently, labeled
$T_0$ topologies) on $19$ points,
\[
P(19)=646{,}099{,}441{,}937{,}791{,}106{,}493{,}755{,}218{,}560{,}442{,}089{,}979
\]
a $39$-digit integer extending OEIS \texttt{A001035}, whose largest previously computed term was $P(18)$
(Brinkmann and McKay~\cite{BrinkmannMcKay}). By the Stirling transform we also obtain the number of labeled topologies on
$19$ points, \texttt{A000798}$(19)=689{,}054{,}943{,}207{,}246{,}404{,}281{,}592{,}791{,}142{,}107{,}048{,}261$.
Our route is the Ern\'e--Stege moment reduction, which expresses $P(19)$ through a few sums of antichain
counts over the posets on at most $16$ points. All of these are available from the posets on at most $15$
points (whose number is catalogued, and which standard software generates on demand), except a single moment
over the $16$-point posets. That moment is obtained not by enumerating the $16$-point posets but by inserting
a single element into the $15$-point ones, with a per-parent kernel that advances the sum at the cost of
computing the parent's own antichain count. The result passes several independent checks, among them the residue
predicted by the modular periodicity of \texttt{A001035} and the recovery from the same sweep of the known
count $P(16)$ and the Ern\'e--Stege moments $G(16,1)$ and $G(16,2)$. We also report the moments $G(16,3)$ and $G(16,4)$, the latter an input to
the analogous computation for $20$ points.
\end{abstract}

\section{Introduction}

Write $P(n)$ for the number of partial orders on a labeled $n$-element set (OEIS \texttt{A001035}); by the
Alexandrov correspondence, this is also the number of $T_0$ topologies on $n$ points. Write $T(n)$ for the
number of all labeled topologies on $n$ points (\texttt{A000798}). The two are related by
the Stirling transform
\begin{equation}\label{eq:stirling}
T(n)=\sum_{k=0}^{n} S(n,k)\,P(k)
\end{equation}
with $S(n,k)$ the Stirling numbers of the second kind. Their growth is super-exponential:
$\log_2 P(n)\sim n^2/4$ (Kleitman and Rothschild~\cite{KR}).

Each advance in the exact values has required a substantial computation.
Ern\'e and Stege~\cite{ErneStege} obtained the labeled counts through $n=14$;
Heitzig and Reinhold~\cite{HeitzigReinhold} generated the unlabeled posets on $14$ points and, through a
formula of Ern\'e, obtained the labeled counts up to $P(16)$ without generating the larger posets; and
Brinkmann and McKay~\cite{BrinkmannMcKay} enumerated the unlabeled posets on $16$ points and obtained the
labeled counts $P(17)$ and $P(18)$. The largest published terms
are thus $P(18)$ for \texttt{A001035}, \texttt{A000112}$(16)$ for the unlabeled posets, and
\texttt{A000798}$(18)$ for the topologies; the labeled frontier has stood at $n=18$ since the enumeration of Brinkmann and McKay~\cite{BrinkmannMcKay}.

Here, we advance the labeled count to $n=19$. The strategy, developed in Section~\ref{sec:method}, is to reduce
$P(19)$ to a small set of moments (sums of antichain counts over smaller posets), only one of which is
not already known, and then to compute that single moment by extending the posets on $15$ points rather
than enumerating those on $16$. A per-parent kernel (Section~\ref{sec:kernel}) makes the $15$-point sweep
affordable. The Stirling transform then gives the topology count $T(19)$, and the same sweep produces the
moments $G(16,3)$ and $G(16,4)$.

\section{Method}\label{sec:method}

\subsection{The Ern\'e--Stege moment reduction}
For a finite poset $Q$, an \emph{order ideal} (or \emph{down-set}) is a subset $D\subseteq Q$ closed
downward: if $x\in D$ and $y\le x$ in $Q$, then $y\in D$. An \emph{antichain} is a subset of $Q$ whose
elements are pairwise incomparable. The two are in bijection: an order ideal corresponds to the antichain
of its maximal elements, and an antichain to the down-set it generates, so a poset has equally many of
each, the common count we denote $\dd(Q)$. The order ideals of $Q$, ordered by inclusion, themselves form
a distributive lattice $\mathcal{J}(Q)$, the \emph{ideal lattice} (meet is intersection, join is union); its size is
$|\mathcal{J}(Q)|=\dd(Q)$.

For non-negative integers $m$ and $k$, the \emph{automorphism-weighted antichain-power moment} $G(m,k)$ is
the sum of $\dd(Q)^k$ over all labeled posets $Q$ on $m$ points:
\begin{equation}\label{eq:moment}
G(m,k)\;=\;\sum_{\substack{Q\ \text{labeled poset}\\ \text{on } [m]}} \dd(Q)^k
\;=\;\sum_{Q\in\mathcal P_m}\frac{m!}{|\Aut Q|}\,\dd(Q)^k
\end{equation}
The second equality regroups these by
isomorphism type. Here $\mathcal P_m$ ranges over the isomorphism classes of $m$-element posets, and
$\Aut Q$ is the \emph{automorphism group} of $Q$: its \emph{symmetries}. A symmetry is a relabeling of the
points that leaves the order relation exactly as it was, so that the relabeled poset is indistinguishable
from the original. We write $|\Aut Q|$ for the number of such symmetries. Relabeling the $m$ points of $Q$ in
all $m!$ possible ways yields exactly $m!/|\Aut Q|$ distinct labeled posets (two relabelings give the same
one precisely when they differ by an automorphism), and all of them share the same value $\dd(Q)$, since a
symmetry permutes the order ideals among themselves. The second form is the one we compute, as it sweeps
the \emph{unlabeled} posets, far fewer than the labeled ones; the labeled posets are never enumerated
individually.

Evaluating the second form of~\eqref{eq:moment} requires visiting each isomorphism class of $m$-point
posets exactly once. The classes are not available in advance as a list; they are produced one
representative at a time by a \emph{generator}.
A poset is built up a point at a time, and among all the ways a given shape can arise during this process
exactly one is accepted, namely the build in which each new point extends a fixed ``canonical'' labeling of
the shape so far, while every other route to the same shape is recognised and discarded as it occurs. This is
the canonical-construction-path, or \emph{orderly}, method of McKay~\cite{McKay98}: it emits each unlabeled
poset once, in a single deterministic stream, and never compares two independently built posets for
isomorphism. That last property is what lets the stream be cut into disjoint pieces and run in parallel
(Section~\ref{sec:sweep}). The canonical-form computation that decides acceptance also returns the
order of the automorphism group $|\Aut Q|$, and hence the labeling weight $m!/|\Aut Q|$
in~\eqref{eq:moment}, at no extra cost. Concretely, for $m=15$ this stream visits the
$A000112(15)\approx6.83\times10^{13}$ unlabeled posets in place
of the $P(15)\approx7.8\times10^{25}$ labeled ones, a factor close to $15!\approx1.3\times10^{12}$
(almost every poset has no symmetry beyond the identity, so it has the full count of $15!$ distinct labelings).
In short, $\mathcal P_m$ is realised as the output of a generator, not as an explicit list, and the weight
$m!/|\Aut Q|$ is read off the same generation step.

These moments enter a reduction of the labeled count~\cite{ErneStege} (recorded on \texttt{A001035}). It
helps to picture the moments $G(m,k)$ as a two-index array, with the number of points $m$ along one axis
and the power $k$ along the other. The reduction that produces $P(N)$ uses only the entries with
$m+k=N$, which lie along a single diagonal of this array; for that reason we call them the \emph{diagonal
moments} $G(m,N{-}m)$ (Figure~\ref{fig:diagonal}). Passing from the target $N{-}1$ to $N$ moves attention from one diagonal to the
next, a step we will repeatedly describe as advancing one diagonal. The identity that reaches $P(N)$ from
$P(N{-}1)$ together with these diagonal moments is
\begin{equation}\label{eq:reduction}
P(N) \;=\; \binom{N{+}1}{2}\,P(N{-}1)\;-\;\sum_{m=0}^{N-3}(-1)^{\,N-m}\binom{N{-}1{-}m}{2}\binom{N}{m}\,G(m,\,N{-}m)
\end{equation}

\begin{figure}[t]
\centering
\begin{tikzpicture}[scale=0.8,>=stealth,
   used/.style={draw=blue!55!black,fill=blue!16,minimum size=11pt,inner sep=0pt},
   excl/.style={draw=gray!45,fill=gray!8,minimum size=11pt,inner sep=0pt},
   front/.style={draw=black,fill=orange!80,minimum size=11pt,inner sep=0pt},
   frontnext/.style={draw=orange!85!black,fill=orange!22,minimum size=11pt,inner sep=0pt}]
  \foreach \m in {0,...,8}{\foreach \k in {0,...,8}{\fill[gray!30] (\m,\k) circle (1.3pt);}}
  \draw[red!75!black,dashed,thick] (-0.4,7.4) -- (7.4,-0.4);
  \draw[blue!55!black,thick] (-0.4,8.4) -- (8.4,-0.4);
  \draw[green!55!black,dashed,thick] (0.6,8.4) -- (8.4,0.6);
  \begin{scope}[shift={(4.45,5.95)}]
    \fill[white,opacity=0.9] (-0.18,-0.32) rectangle (3.25,2.05);
    \draw[green!55!black,dashed,thick] (0,1.7) -- (0.6,1.7);
       \node[anchor=west,font=\scriptsize] at (0.66,1.7) {$m{+}k=N{+}1$};
    \draw[blue!55!black,thick] (0,0.85) -- (0.6,0.85);
       \node[anchor=west,font=\scriptsize] at (0.66,0.85) {$m{+}k=N$};
    \draw[red!75!black,dashed,thick] (0,0.0) -- (0.6,0.0);
       \node[anchor=west,font=\scriptsize] at (0.66,0.0) {$m{+}k=N{-}1$};
  \end{scope}
  \foreach \m/\k in {0/8,1/7,2/6,3/5,4/4}{\node[used] at (\m,\k) {};}
  \node[front] (fr) at (5,3) {};
  \node[frontnext] (frn) at (6,3) {};
  \node[orange!55!black,font=\scriptsize,align=left,anchor=west] at (6.5,3.0)
       {$G(N{-}2,3)$:\\ next frontier};
  \foreach \m/\k in {6/2,7/1,8/0}{\node[excl] at (\m,\k) {};}
  \draw[->,very thick,orange!75!black] (fr.east) -- (frn.west);
  \draw[->] (-0.5,0) -- (9.6,0) node[right] {$m$ (points)};
  \draw[->] (0,-0.5) -- (0,9.4) node[above] {$k$ (power)};
  \foreach \i in {0,...,8}{\node[below] at (\i,-0.45) {\scriptsize \i};
                           \node[left] at (-0.45,\i) {\scriptsize \i};}
  \node[align=center] at (4.0,-2.15)
     {\footnotesize frontier $G(N{-}3,3)$ for $P(N)$\\ \footnotesize (the hardest term: largest $m$ on the diagonal)};
  \draw[->,thick] (4.0,-1.55) -- (4.85,2.7);
\end{tikzpicture}
\caption{The moment array $G(m,k)$, indexed by the number of points $m$ and the power $k$. For a target
$N$, the reduction~\eqref{eq:reduction} sums only the moments on the diagonal $m+k=N$; the figure is drawn
for the small target $N=8$. The filled cells (those with $k\ge3$) are the terms that actually appear, and
the lowest of them, $G(N{-}3,3)$ (orange), is the computational frontier. The faded cells ($k<3$) are
absent from the sum. The two dashed lines are the neighbouring diagonals $m+k=N{-}1$ (red) and
$m+k=N{+}1$ (green).
Because the frontier always sits at $k=3$, raising the target by one slides it one step to the right along
that row, from $G(N{-}3,3)$ to $G(N{-}2,3)$ (open orange), and shifts the whole problem outward to the next
diagonal: this is what we mean by advancing one diagonal. In the present work $N=19$, so the diagonal runs
from $G(0,19)$ down to the frontier $G(16,3)$, and every entry with $m\le14$ is already known and absorbed
into the constant $A$ of~\eqref{eq:Adef}.}
\label{fig:diagonal}
\end{figure}
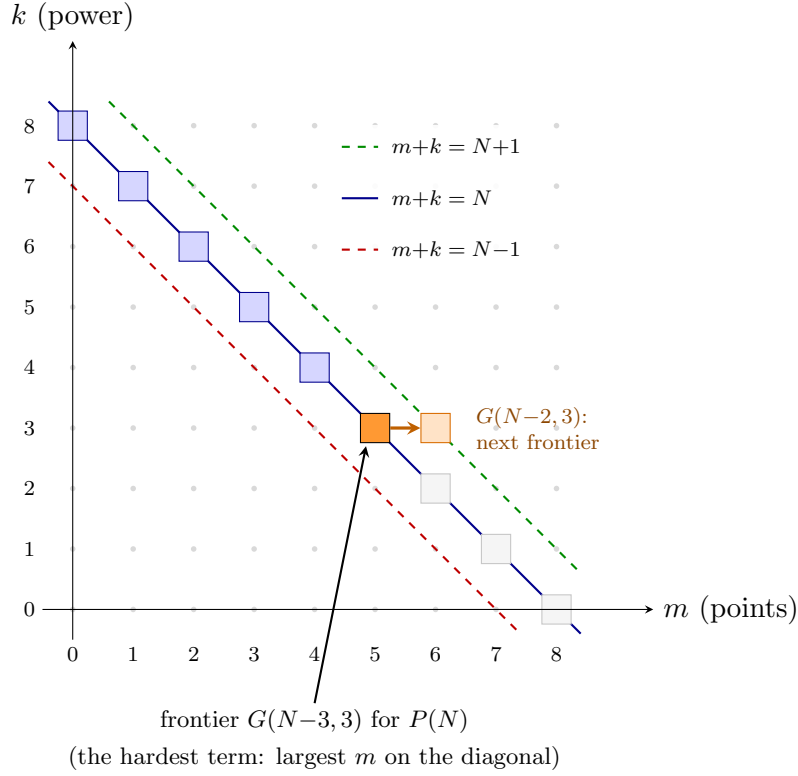

Setting $N=19$ and writing out the two largest terms of the sum (the $m=16$ term, contributing
$+\binom{19}{3}G(16,3)=+969\,G(16,3)$, and the $m=15$ term, contributing
$-\binom{3}{2}\binom{19}{4}G(15,4)=-11628\,G(15,4)$), this becomes
\begin{equation}\label{eq:p19full}
P(19) \;=\; 190\,P(18)\;+\;969\,G(16,3)\;-\;11628\,G(15,4)
\;-\!\!\sum_{m=0}^{14}(-1)^{\,19-m}\binom{18{-}m}{2}\binom{19}{m}\,G(m,19{-}m)
\end{equation}
On the right, only the count $P(18)$ is a previously published value; every other term is a moment
computed in this work. Each lower diagonal moment $G(m,19{-}m)$ with $m\le 14$ is, by its
definition~\eqref{eq:moment}, a sum of $\dd(Q)^{19-m}$ over the labeled posets $Q$ on $m$ points, which we
evaluate by directly sweeping those posets (tallying the antichain count of each). Gathering $P(18)$ and these $m\le14$
moments into one constant,
\begin{equation}\label{eq:Adef}
A \;:=\; 190\,P(18)\;-\!\!\sum_{m=0}^{14}(-1)^{\,19-m}\binom{18{-}m}{2}\binom{19}{m}\,G(m,19{-}m)
\end{equation}
the identity becomes simply
\begin{equation}\label{eq:p19}
P(19) \;=\; A \;-\; 11628\,G(15,4)\;+\;969\,G(16,3)
\end{equation}
with the constant and the supporting moment given by
\begin{align*}
A &= 5{,}325{,}468{,}436{,}052{,}842{,}213{,}619{,}347{,}019{,}464{,}238{,}237{,}629\\
G(15,4) &= 846{,}002{,}793{,}378{,}179{,}474{,}085{,}677{,}125{,}510{,}787{,}278
\end{align*}
Both remaining moments are computed at the same, $15$-point, scale. $G(15,4)$ is a direct sweep of the
catalogued $15$-point posets (summing $\dd(Q)^4$ over them), while the $16$-point moment $G(16,3)$ is
harvested by inserting a single element into those very same posets (Section~\ref{sec:harvest}), forming
no $16$-point poset. $G(16,3)$ is the frontier moment, and obtaining it occupies the rest of this section. With both in hand, $P(19)$ follows
from~\eqref{eq:p19} by a single multiplication and addition. (\eqref{eq:reduction} likewise reproduces the
known $P(17)$ and $P(18)$, which need only moments through the $15$-point posets; $P(19)$ is the first to
require one from the $16$-point level.)

\subsection{Isomorphism-free harvest of the 16-point moment}\label{sec:harvest}
The remaining task is to evaluate $G(16,3)$ (and, in the same sweep, $G(16,k)$ for $k\le 4$). Generating
all roughly $4.48\times10^{15}$ unlabeled $16$-point posets is infeasible. Instead we harvest the $16$-point
moment from the $A000112(15)=68{,}275{,}077{,}901{,}156$ unlabeled $15$-point posets by one-point
insertion (Figure~\ref{fig:insertion}). This follows Heitzig and Reinhold~\cite{HeitzigReinhold}, who
likewise harvest a labeled count one diagonal above their generated objects with an algorithm that
``avoids isomorphism tests and can therefore be parallelized.'' Throughout, $Q+z$ denotes the \emph{child}
formed by inserting a new point $z$ into $Q$. Such an insertion is \emph{admissible} when the elements
below $z$ form a down-set, those above an up-set, and every element below lies under every element above.

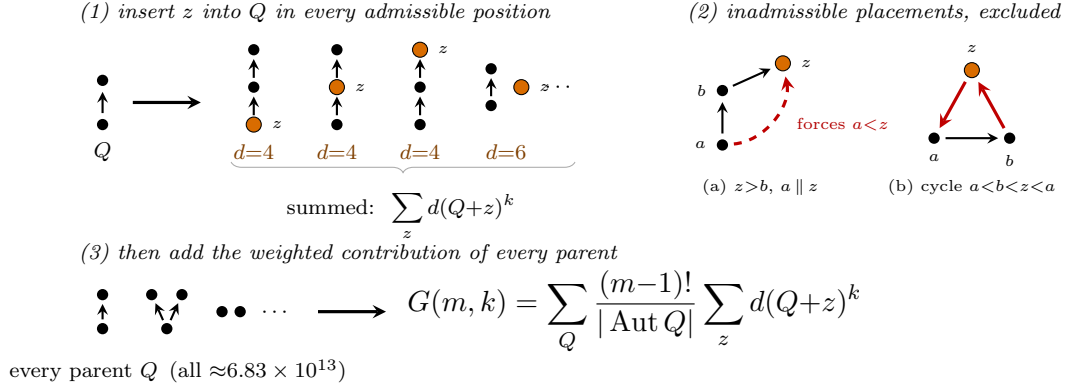
\begin{figure}[tbp]
\centering
\begin{tikzpicture}[x=1cm,y=0.9cm,>=stealth,
   pt/.style={circle,fill=black,inner sep=1.5pt},
   zz/.style={circle,fill=orange!85!black,draw=black,inner sep=2pt},
   ed/.style={thick},
   gv/.style={->,thick,shorten >=2pt,shorten <=2pt},
   gvr/.style={->,thick,shorten >=3.5pt,shorten <=3.5pt},
   bad/.style={->,very thick,red!75!black,shorten >=2pt,shorten <=2pt},
   rl/.style={font=\scriptsize},
   dl/.style={font=\scriptsize,orange!50!black},
   hd/.style={font=\scriptsize\itshape},
   ti/.style={font=\tiny}]
  \node[hd,anchor=west] at (-0.45,3.2){(1) insert $z$ into $Q$ in every admissible position};
  \node[pt] (pa) at (0,1.55){}; \node[pt] (pb) at (0,2.20){};
  \draw[gv] (pa)--(pb); \node[rl] at (0,1.15){$Q$};
  \draw[->,very thick] (0.40,1.875) -- (1.30,1.875);
  \begin{scope}[shift={(2.0,1.55)}]
    \node[zz,label=right:{\tiny$z$}] at (0,0){}; \node[pt] at (0,0.55){}; \node[pt] at (0,1.10){};
    \draw[gvr] (0,0)--(0,0.55); \draw[gvr] (0,0.55)--(0,1.10); \node[dl] at (0,-0.40){$\dd{=}4$};\end{scope}
  \begin{scope}[shift={(3.1,1.55)}]
    \node[pt] at (0,0){}; \node[zz,label=right:{\tiny$z$}] at (0,0.55){}; \node[pt] at (0,1.10){};
    \draw[gvr] (0,0)--(0,0.55); \draw[gvr] (0,0.55)--(0,1.10); \node[dl] at (0,-0.40){$\dd{=}4$};\end{scope}
  \begin{scope}[shift={(4.2,1.55)}]
    \node[pt] at (0,0){}; \node[pt] at (0,0.55){}; \node[zz,label=right:{\tiny$z$}] at (0,1.10){};
    \draw[gvr] (0,0)--(0,0.55); \draw[gvr] (0,0.55)--(0,1.10); \node[dl] at (0,-0.40){$\dd{=}4$};\end{scope}
  \begin{scope}[shift={(5.3,1.55)}]
    \node[pt] at (-0.16,0.28){}; \node[pt] at (-0.16,0.82){}; \draw[gvr] (-0.16,0.28)--(-0.16,0.82);
    \node[zz,label=right:{\tiny$z$}] at (0.24,0.55){}; \node[dl] at (0.04,-0.40){$\dd{=}6$};\end{scope}
  \node[rl] at (6.05,2.1){$\cdots$};
  \draw[decorate,decoration={brace,amplitude=4pt,mirror},gray!65] (1.7,1.0) -- (6.25,1.0);
  \node[rl,anchor=north] at (3.95,0.72){summed: $\;\displaystyle\sum_{z}\dd(Q{+}z)^k$};
  \node[hd,anchor=west] at (7.6,3.2){(2) inadmissible placements, excluded};
  \node[pt,label=left:{\tiny$a$}] (a) at (8.2,1.25){};
  \node[pt,label=left:{\tiny$b$}] (b) at (8.2,2.05){};
  \node[zz,label=right:{\tiny$z$}] (z) at (9.0,2.45){};
  \draw[gv] (a)--(b); \draw[gv] (b)--(z);
  \draw[bad,dashed] (a) to[bend right=48] (z);
  \node[ti,red!70!black,anchor=west] at (9.05,1.55){forces $a{<}z$};
  \node[ti,anchor=north] at (8.7,0.9){(a) $z{>}b$, $a\,\|\,z$};
  \node[pt,label=below:{\tiny$a$}] (a2) at (11.0,1.35){};
  \node[pt,label=below:{\tiny$b$}] (b2) at (12.0,1.35){};
  \node[zz,label=above:{\tiny$z$}] (z2) at (11.5,2.35){};
  \draw[gv] (a2)--(b2); \draw[bad] (b2)--(z2); \draw[bad] (z2)--(a2);
  \node[ti,anchor=north] at (11.5,0.9){(b) cycle $a{<}b{<}z{<}a$};
  \node[hd,anchor=west] at (-0.45,-0.35){(3) then add the weighted contribution of every parent};
  \node[pt] at (0,-1.45){}; \node[pt] at (0,-0.95){}; \draw[gvr] (0,-1.45)--(0,-0.95);
  \begin{scope}[shift={(0.85,0)}]
    \node[pt] at (0,-1.45){}; \node[pt] at (-0.2,-0.95){}; \node[pt] at (0.2,-0.95){};
    \draw[gvr] (0,-1.45)--(-0.2,-0.95); \draw[gvr] (0,-1.45)--(0.2,-0.95);\end{scope}
  \begin{scope}[shift={(1.7,0)}]
    \node[pt] at (-0.12,-1.2){}; \node[pt] at (0.12,-1.2){};\end{scope}
  \node[rl] at (2.3,-1.2){$\cdots$};
  \node[rl,anchor=north] at (1.0,-1.75){every parent $Q$ \,(all ${\approx}6.83\times10^{13}$)};
  \draw[->,very thick] (2.85,-1.2) -- (3.75,-1.2);
  \node[anchor=west] at (3.9,-1.2)
     {$\displaystyle G(m,k)=\sum_{Q}\frac{(m{-}1)!}{|\Aut Q|}\sum_{z}\dd(Q{+}z)^k$};
\end{tikzpicture}
\caption{Harvesting the $m$-point moment by one-point insertion (throughout, an arrow $x\to y$ means
$x<y$). \textbf{(1)} Into each unlabeled parent
$Q$ on $m{-}1$ points, a new point $z$ (orange) is inserted in \emph{every} admissible position (below,
within, or above the existing order, or incomparable beside it); the antichain count $\dd(Q{+}z)$ of each
child (shown beneath) is gathered into the per-parent sum $\sum_z\dd(Q{+}z)^k$. \textbf{(2)} Placements that
violate admissibility are excluded (forced or contradictory
relations in red): putting $z$ above $b$ while leaving $a$ incomparable would force
$a<z$ by transitivity, and putting $z$ above $b$ but below $a$ would close the cycle $a<b<z<a$. \textbf{(3)}
The per-parent sums are then added over \emph{all} parents, each weighted by $w=(m{-}1)!/|\Aut Q|$, giving
$G(m,k)=\sum_Q w\sum_z\dd(Q{+}z)^k$. In the computation $m=16$ and the parents are the roughly
$6.83\times10^{13}$ unlabeled $15$-point posets, each swept once; the $16$-point posets are never listed.}
\label{fig:insertion}
\end{figure}

The basis is a deletion correspondence. Every labeled $16$-poset has a unique element of largest label;
deleting it leaves a labeled $15$-poset $Q'$ together with the record of how the deleted element sat in
it, that is, the set of elements that lay below it and the set that lay above. Conversely, $Q'$ and any such
admissible record rebuild the original poset uniquely, by reinstating the deleted element with exactly
those relations. Deletion is therefore a bijection between the labeled $16$-posets and the pairs
(labeled $15$-poset, admissible insertion), so summing $\dd^k$ over the former equals summing the child
count $\dd(Q'+z)^k$ over the latter. Two labelings of the same $15$-point shape admit the same insertions
and produce children with the same antichain counts, so the labeled $15$-posets may be grouped by
isomorphism type: a type $Q$ contributes its $15!/|\Aut Q|$ distinct labelings, all with identical
insertion data. The sum thus collapses to one over unlabeled parents $Q$, each weighted by
$w=15!/|\Aut Q|$ and ranging over \emph{all} admissible insertions, and this reconstructs $G(16,k)$ exactly.
For general $m$ this is the standard deletion correspondence~\cite{HeitzigReinhold,BrinkmannMcKay},
\begin{equation}\label{eq:harvest}
G(m,k)\;=\;\sum_{Q}\frac{(m{-}1)!}{|\Aut Q|}\sum_{z}\dd(Q+z)^k,
\end{equation}
the outer sum over the unlabeled $(m{-}1)$-point posets $Q$ and the inner over the admissible insertions $z$;
Proposition~\ref{prop:kernel} below evaluates the inner sum without ever forming a child.

An insertion of a new element $z$ into $Q$ records how $z$ relates to each existing element: every element
is placed \emph{below} $z$, placed \emph{above} $z$, or left \emph{incomparable} to it. The elements below
$z$ form an order ideal $D$ of $Q$, those above form a filter $U$ disjoint from $D$, and the placement is
admissible exactly when every element of $D$ lies below every element of $U$ in $Q$ (otherwise transitivity
would force a further relation; see Figure~\ref{fig:insertion}(2)). Notice that the inserted element is
maximal precisely when $U=\varnothing$. Write $J=Q\setminus U$ for the order ideal of elements \emph{not} above $z$, namely
those below $z$ together with those incomparable to it, a single element of the ideal lattice
$\mathcal{J}(Q)$.

To count the order ideals (equivalently antichains) of the child $Q+z$, we split them by whether they
contain the new point $z$ (Figure~\ref{fig:idealsplit}).

\begin{lemma}\label{lem:childcount}
Write $c_{\mathrm{sub}}(J)=\#\{\text{ideals}\subseteq J\}$ and $c_{\sup}(D)=\#\{\text{ideals}\supseteq D\}$ for the
numbers of order ideals of $Q$ contained in $J$ and containing $D$, respectively. Then the antichain count
of the child is
\begin{equation}\label{eq:childcount}
\dd(Q+z) \;=\; c_{\mathrm{sub}}(J) \;+\; c_{\sup}(D).
\end{equation}
\end{lemma}
\begin{proof}
If an
ideal \emph{omits} $z$, it can hold no element lying above $z$, since a down-set containing such an element
would have to contain $z$ as well; the ideal therefore lies inside $J$, and these are exactly the order
ideals of $J$, numbering $c_{\mathrm{sub}}(J)$. If instead an ideal \emph{contains} $z$, then being downward
closed it contains all of $D$, and conversely every order ideal of $Q$ that contains $D$ becomes one of
these by adjoining $z$; these number $c_{\sup}(D)$. The two cases are disjoint and exhaustive.
\end{proof}
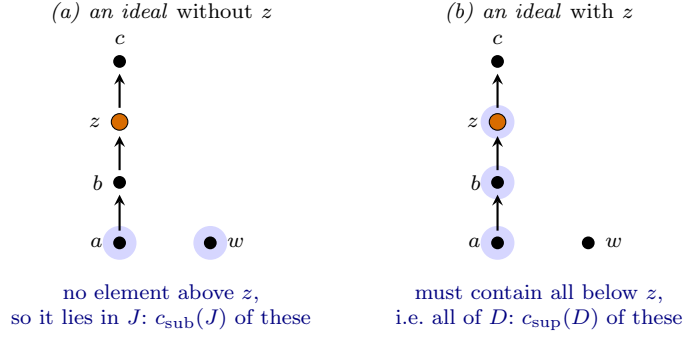
\begin{figure}[tbp]
\centering
\begin{tikzpicture}[x=1cm,y=1cm,>=stealth,
   pt/.style={circle,fill=black,inner sep=1.7pt},
   zp/.style={circle,fill=orange!85!black,draw=black,inner sep=2.2pt},
   gv/.style={->,thick,shorten >=2pt,shorten <=2pt},
   halo/.style={circle,fill=blue!16,inner sep=4.5pt},
   ti/.style={font=\scriptsize\itshape},
   tg/.style={font=\scriptsize,blue!50!black}]
  \begin{scope}
    \node[ti] at (0.55,3.05){(a) an ideal \emph{without} $z$};
    \node[halo] at (0,0){}; \node[halo] at (1.2,0){};
    \node[pt,label=left:{\scriptsize$a$}] (a) at (0,0){};
    \node[pt,label=left:{\scriptsize$b$}] (b) at (0,0.8){};
    \node[zp,label=left:{\scriptsize$z$}] (z) at (0,1.6){};
    \node[pt,label=above:{\scriptsize$c$}] (c) at (0,2.4){};
    \node[pt,label=right:{\scriptsize$w$}] (w) at (1.2,0){};
    \draw[gv] (a)--(b); \draw[gv] (b)--(z); \draw[gv] (z)--(c);
    \node[tg,align=center] at (0.55,-0.85){no element above $z$,\\ so it lies in $J$:\ $c_{\mathrm{sub}}(J)$ of these};
  \end{scope}
  \begin{scope}[shift={(5.0,0)}]
    \node[ti] at (0.55,3.05){(b) an ideal \emph{with} $z$};
    \node[halo] at (0,0){}; \node[halo] at (0,0.8){}; \node[halo] at (0,1.6){};
    \node[pt,label=left:{\scriptsize$a$}] (a2) at (0,0){};
    \node[pt,label=left:{\scriptsize$b$}] (b2) at (0,0.8){};
    \node[zp,label=left:{\scriptsize$z$}] (z2) at (0,1.6){};
    \node[pt,label=above:{\scriptsize$c$}] (c2) at (0,2.4){};
    \node[pt,label=right:{\scriptsize$w$}] (w2) at (1.2,0){};
    \draw[gv] (a2)--(b2); \draw[gv] (b2)--(z2); \draw[gv] (z2)--(c2);
    \node[tg,align=center] at (0.55,-0.85){must contain all below $z$,\\ i.e.\ all of $D$:\ $c_{\sup}(D)$ of these};
  \end{scope}
\end{tikzpicture}
\caption{Counting the order ideals (antichains) of a child $Q+z$ by whether they contain the new point $z$
(an arrow $x\to y$ means $x<y$), shown for the child with $a<b<z<c$ and $w$ incomparable to the rest. Here
$D=\{a,b\}$ are the elements below $z$, $U=\{c\}$ those above, $w$ is incomparable, and
$J=Q\setminus U=\{a,b,w\}$ collects everything not above $z$. Shaded nodes form the chosen ideal. \textbf{(a)} An ideal omitting $z$ can hold
nothing above $z$ (that would drag $z$ in), so it sits inside $J$ and is simply an ideal of $J$: there are
$c_{\mathrm{sub}}(J)$ of them. \textbf{(b)} An ideal containing $z$ must contain all of $D$, and every ideal
of $Q$ containing $D$ arises this way: there are $c_{\sup}(D)$ of them. The two kinds are disjoint and
exhaustive, giving $\dd(Q+z)=c_{\mathrm{sub}}(J)+c_{\sup}(D)$ as in~\eqref{eq:childcount}.}
\label{fig:idealsplit}
\end{figure}

Recall that an insertion is specified by the pair $(J,D)$ of order ideals introduced above ($J$ the elements
\emph{not} above $z$, $D$ those \emph{below} it). To range over all insertions we let $J$ run over the order ideals of $Q$,
and for each fixed $J$ we let $D$ run over the order ideals contained in
\begin{equation}\label{eq:bstar}
B^\ast(J)\;:=\;J\cap\!\!\bigcap_{u\in Q\setminus J}\!{\downarrow}u,\qquad {\downarrow}u=\{x\in Q:x\le u\}
\end{equation}
In~\eqref{eq:bstar} the index $u$ runs over the elements of $U=Q\setminus J$ (those placed above $z$);
${\downarrow}u$ is the \emph{principal down-set} of $u$, the elements at or below it; and
the inner intersection $\bigcap_{u\in U}{\downarrow}u$ keeps the elements that lie below \emph{every} member
of $U$, which the outer $J\cap{}$ then restricts to $J$. So $B^\ast(J)$ is the part of $J$ sitting below
all of $U$ (Figure~\ref{fig:bstar}). Notice that the requirement $D\subseteq B^\ast(J)$
is exactly the admissibility condition: anything placed below $z$ must lie below everything placed above
$z$, or transitivity would force a new relation. For instance, in Figure~\ref{fig:bstar} the element $b$ lies in $J=\{a,b\}$ but not in $B^\ast(J)=\{a\}$,
because it sits below $p$ but not below $q$ and so cannot be placed below $z$. When $z$ is maximal,
$U=\varnothing$, the intersection is over no elements and equals $Q$, so $B^\ast(Q)=Q$ and $D$ may be any
order ideal.

\begin{figure}[tbp]
\centering
\begin{tikzpicture}[x=1cm,y=1cm,>=stealth,
   pt/.style={circle,fill=black,inner sep=1.7pt},
   uu/.style={circle,fill=green!60!black,draw=black,inner sep=2.2pt},
   zz/.style={circle,fill=orange!85!black,draw=black,inner sep=2.2pt},
   gv/.style={->,thick,shorten >=2pt,shorten <=2pt},
   halo/.style={circle,fill=blue!16,inner sep=4.5pt},
   ti/.style={font=\scriptsize\itshape},
   op/.style={font=\large},
   lbl/.style={font=\scriptsize\bfseries}]
  \node[lbl] at (-1.95,0.45){(1)};
  \begin{scope}
    \node[ti] at (-0.3,1.6){$J=\{a,b\}$};
    \node[halo] at (0,0){}; \node[halo] at (-1.0,0){};
    \node[pt,label=below:{\scriptsize$a$}] (a) at (0,0){};
    \node[pt,label=below:{\scriptsize$b$}] (b) at (-1.0,0){};
    \node[uu,label=above:{\scriptsize$p$}] (p) at (-0.45,0.9){};
    \node[uu,label=above:{\scriptsize$q$}] (q) at (0.45,0.9){};
    \draw[gv] (a)--(p); \draw[gv] (a)--(q); \draw[gv] (b)--(p);
  \end{scope}
  \node[op] at (0.97,0.45){$\cap$};
  \begin{scope}[shift={(2.5,0)}]
    \node[ti] at (-0.3,1.6){${\downarrow}p$};
    \node[halo] at (0,0){}; \node[halo] at (-1.0,0){}; \node[halo] at (-0.45,0.9){};
    \node[pt,label=below:{\scriptsize$a$}] (a2) at (0,0){};
    \node[pt,label=below:{\scriptsize$b$}] (b2) at (-1.0,0){};
    \node[uu,label=above:{\scriptsize$p$}] (p2) at (-0.45,0.9){};
    \node[uu,label=above:{\scriptsize$q$}] (q2) at (0.45,0.9){};
    \draw[gv] (a2)--(p2); \draw[gv] (a2)--(q2); \draw[gv] (b2)--(p2);
  \end{scope}
  \node[op] at (3.47,0.45){$\cap$};
  \begin{scope}[shift={(5.0,0)}]
    \node[ti] at (-0.3,1.6){${\downarrow}q$};
    \node[halo] at (0,0){}; \node[halo] at (0.45,0.9){};
    \node[pt,label=below:{\scriptsize$a$}] (a3) at (0,0){};
    \node[pt,label=below:{\scriptsize$b$}] (b3) at (-1.0,0){};
    \node[uu,label=above:{\scriptsize$p$}] (p3) at (-0.45,0.9){};
    \node[uu,label=above:{\scriptsize$q$}] (q3) at (0.45,0.9){};
    \draw[gv] (a3)--(p3); \draw[gv] (a3)--(q3); \draw[gv] (b3)--(p3);
  \end{scope}
  \node[op] at (5.97,0.45){$=$};
  \begin{scope}[shift={(7.5,0)}]
    \node[ti] at (-0.3,1.6){$B^\ast(J)=\{a\}$};
    \node[halo] at (0,0){};
    \node[pt,label=below:{\scriptsize$a$}] (a4) at (0,0){};
    \node[pt,label=below:{\scriptsize$b$}] (b4) at (-1.0,0){};
    \node[uu,label=above:{\scriptsize$p$}] (p4) at (-0.45,0.9){};
    \node[uu,label=above:{\scriptsize$q$}] (q4) at (0.45,0.9){};
    \draw[gv] (a4)--(p4); \draw[gv] (a4)--(q4); \draw[gv] (b4)--(p4);
  \end{scope}
  \draw[decorate,decoration={brace,amplitude=4pt,mirror},gray!65] (1.3,-0.6) -- (5.65,-0.6);
  \node[font=\scriptsize] at (3.47,-1.05){$\bigcap_{u\in U}{\downarrow}u$, \ $u=p,q$};
  \node[lbl,anchor=west] at (-1.95,-1.75){(2) the two insertions, $D\subseteq B^\ast(J)=\{a\}$:};
  \begin{scope}[shift={(0.6,-4.3)}]
    \node[pt,label=below:{\scriptsize$b$}] (xb) at (-1.35,0){};
    \node[pt,label=below:{\scriptsize$a$}] (xa) at (-0.4,0){};
    \node[zz,label=below:{\scriptsize$z$}] (xz) at (0.4,0){};
    \node[uu,label=above:{\scriptsize$p$}] (xp) at (-0.4,1.0){};
    \node[uu,label=above:{\scriptsize$q$}] (xq) at (0.4,1.0){};
    \draw[gv] (xb)--(xp); \draw[gv] (xa)--(xp); \draw[gv] (xa)--(xq);
    \draw[gv] (xz)--(xp); \draw[gv] (xz)--(xq);
    \node[ti] at (-0.3,-0.7){$D=\varnothing$};
  \end{scope}
  \begin{scope}[shift={(4.1,-4.3)}]
    \node[pt,label=below:{\scriptsize$a$}] (ya) at (0,0){};
    \node[zz,label=right:{\scriptsize$z$}] (yz) at (0,0.8){};
    \node[pt,label=left:{\scriptsize$b$}] (yb) at (-1.1,0.8){};
    \node[uu,label=above:{\scriptsize$p$}] (yp) at (-0.5,1.6){};
    \node[uu,label=above:{\scriptsize$q$}] (yq) at (0.5,1.6){};
    \draw[gv] (ya)--(yz); \draw[gv] (yz)--(yp); \draw[gv] (yz)--(yq); \draw[gv] (yb)--(yp);
    \node[ti] at (-0.2,-0.7){$D=\{a\}$};
  \end{scope}
\end{tikzpicture}
\caption{\textbf{(1)} Building $B^\ast(J)=J\cap\bigcap_{u\in U}{\downarrow}u$ of~\eqref{eq:bstar} on the
parent $Q=\{a,b,p,q\}$ with $a<p$, $a<q$, $b<p$ (arrows point from smaller to larger). With $p$ and $q$
(green) above $z$, $U=\{p,q\}$ and $J=\{a,b\}$. The big intersection runs over the two values $u=p$ and
$u=q$, giving the down-sets ${\downarrow}p=\{a,b,p\}$ and ${\downarrow}q=\{a,q\}$ (braced); these are not
nested, and intersecting them, then with $J$, leaves $B^\ast(J)=\{a\}$ (every element but $a$ is removed by
some set: $q\notin{\downarrow}p$, $b,p\notin{\downarrow}q$, $p,q\notin J$). \textbf{(2)} The admissible
insertions for this $U$ are the order ideals $D\subseteq B^\ast(J)=\{a\}$, namely $D=\varnothing$ and
$D=\{a\}$; their two children $Q+z$ (new point $z$ in orange) are shown. The would-be choice $D=\{b\}$ is
excluded because $b<z<q$ would force $b<q$, which is false.}
\label{fig:bstar}
\end{figure}
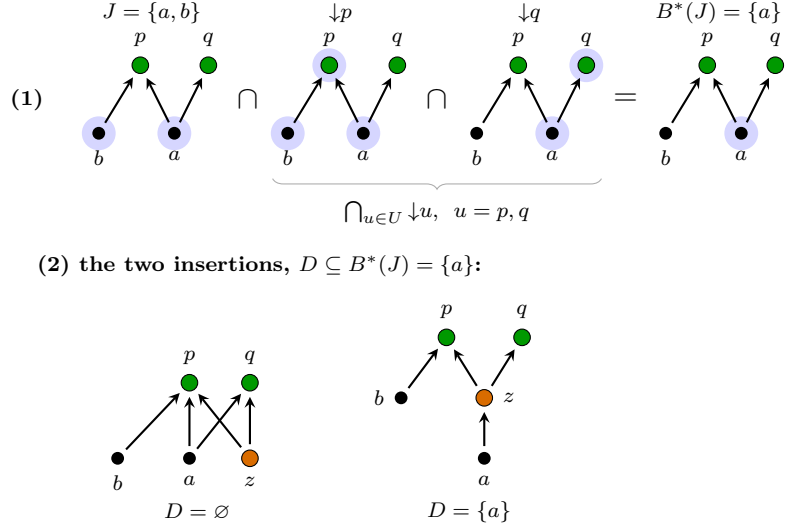

Computationally, $B^\ast(J)$ is not rebuilt for each ideal: adjacent
ideals in $\mathcal{J}(Q)$ differ by a single element $x$ of $Q$ (a \emph{covering arc} of the lattice),
across which the intersection in~\eqref{eq:bstar} changes by just the one principal down-set ${\downarrow}x$;
storing that one arc per ideal threads $B^\ast$ from each ideal to the next, with no full intersection ever
recomputed. Every parent is generated once by
canonical-construction-path generation~\cite{BrinkmannMcKay} with isomorphism handling by
\texttt{nauty}~\cite{nauty}, which also supplies $|\Aut Q|$. The contributions of all insertions to
$G(16,k)$ are accumulated directly from $Q$'s ideal lattice; no $16$-point poset is generated or
canonicalized, and the roughly $8.3\times10^{28}$ labeled posets are never enumerated.

\paragraph{Example.} We trace the entire per-parent computation on the smallest possible parent, a single point, advancing one diagonal from it. Take
the one-point parent $Q=\{a\}$, whose ideal lattice is the two-element chain
$\mathcal{J}(Q)=\{\varnothing,\{a\}\}$, so $\dd(Q)=2$. It admits three insertions of a new element $z$, listed below
by their $(J,D)$ data; here $B^\ast(J)=J$, so $D$ runs over the ideals contained in $J$, and the child
count $\dd=c_{\mathrm{sub}}(J)+c_{\sup}(D)$ is read off~\eqref{eq:childcount}.
\begin{center}
\begin{tabular}{llccc}
\toprule
insertion $(J,D)$ & child & $c_{\mathrm{sub}}(J)$ & $c_{\sup}(D)$ & $\dd$\\
\midrule
$(\{a\},\varnothing)$ & $a,z$ incomparable & $2$ & $2$ & $4$\\
$(\{a\},\{a\})$ & $a<z$ & $2$ & $1$ & $3$\\
$(\varnothing,\varnothing)$ & $z<a$ & $1$ & $2$ & $3$\\
\bottomrule
\end{tabular}
\end{center}
The single parent has weight $w=1!/|\Aut Q|=1$, so its contribution to $G(2,k)$ is $4^k+3^k+3^k$: at
$k=0$ this is $3=P(2)$, the three labeled $2$-point posets; at $k=1$ it is $G(2,1)=10$; at $k=2$ it is
$G(2,2)=34$. The inner double loop of Algorithm~\ref{alg:kernel} forms exactly this insertion sum, but
evaluates the sum over $D$ by the binomial collapse~\eqref{eq:collapse} rather than by listing insertions,
which is what holds the cost to a single sweep of the parent's ideal lattice when $\dd$ is large.

\subsection{A single-pass moment-transfer kernel}\label{sec:kernel}
The kernel evaluates the entire $(m{+}1)$-point moment from the $m$-point sweep in one pass per parent.
Rather than recording each parent's own antichain count and advancing the diagonal through the analytic
reduction~\cite{ErneStege,HeitzigReinhold} or instrumenting generation one level higher~\cite{BrinkmannMcKay},
it reads the moment of the parent's one-point insertions directly off the parent's ideal lattice.

The \emph{zeta transform} of a function $f$ on a finite poset comes in two directions, both used below.
The \emph{down-zeta} $\textsc{DownZeta}$ (written $\zeta$) assigns to each element $x$ the cumulative sum
$(\zeta f)(x)=\sum_{y\le x}f(y)$ over everything at or below $x$; the \emph{up-zeta} $\textsc{UpZeta}$ assigns
the cumulative sum $\sum_{y\ge x}f(y)$ over everything at or above $x$ (each is inverted by the corresponding
M\"obius transform). On the ideal lattice $\mathcal{J}(Q)$, ordered by inclusion, ``at or below $x$'' means the ideals
contained in $x$ and ``at or above $x$'' the ideals containing it, so $\textsc{DownZeta}$ sums $f$ over all
sub-ideals and $\textsc{UpZeta}$ over all super-ideals (Figure~\ref{fig:idealzeta}). Either is computed by a single sweep along a
\emph{linear extension} of $\mathcal{J}(Q)$ (a listing of all the ideals in some order refining inclusion, so each ideal
comes after every ideal it contains): sweeping along such a list in the matching direction, each step already
has the partial sums it needs, which is what makes the transform fast.

\begin{proposition}\label{prop:kernel}
Let $Q$ be an $m$-point poset with ideal lattice $\mathcal{J}(Q)$, and write $\dd=|\mathcal{J}(Q)|$ for its number of order
ideals (equivalently its antichain count). For each $k$, let $S_k(Q)$ be the sum of the child antichain
counts $\dd(Q+z)^k$ over all admissible one-point insertions $z$ into $Q$ (Section~\ref{sec:harvest}):
\[
   S_k(Q)\;=\;\sum_{z}\dd(Q+z)^k
\]
Then $S_0(Q),\dots,S_K(Q)$ can be computed in $O(\dd\,m\,K)$ arithmetic
operations and $O(\dd)$ words of working memory, using only zeta transforms on $\mathcal{J}(Q)$ and without
enumerating the insertions $z$ individually, generating, canonicalizing, or storing any $(m{+}1)$-point
poset.
\end{proposition}

\noindent The full moments are then recovered by the weighting of Section~\ref{sec:harvest},
\[
G(m{+}1,k)=\sum_Q \frac{m!}{|\Aut Q|}\,S_k(Q)
\]
summed over the unlabeled parents $Q$, so the proposition bounds the cost of the
whole sweep. The point is not that the $K{+}1$ power-sums share their inputs (they do, trivially), but that those inputs
are never formed. The number of admissible insertions $z$ can grow as $\Theta(\dd^2)$ (it
is $\dd(\dd{+}1)/2$ already for a chain), so listing them and summing their child counts $\dd(Q+z)^k$ is an
$O(\dd^2)$ computation. The kernel instead obtains every $S_k$ in $O(\dd\,m)$, never forming a single
$\dd(Q+z)$: the binomial collapse~\eqref{eq:collapse} below replaces the inner sum over the (up to $\dd$)
insertions sharing each ideal $J$ by one value read from a zeta transform, leaving a sum over the $\dd$
ideals rather than the $\sim\dd^2$ insertions (Figure~\ref{fig:collapse}).

\begin{proof}
The argument rests on three facts.

\begin{figure}[tbp]
\centering
\begin{tikzpicture}[x=1cm,y=1cm,>=stealth,
   pt/.style={circle,fill=black,inner sep=1.7pt},
   nd/.style={circle,fill=black,inner sep=1.6pt},
   gv/.style={->,thick,shorten >=2pt,shorten <=2pt},
   le/.style={semithick,gray!58,->,shorten >=1.5pt,shorten <=1.5pt},
   halo/.style={circle,fill=blue!16,inner sep=3.8pt},
   halou/.style={circle,fill=green!24,inner sep=3.8pt},
   ti/.style={font=\scriptsize\itshape},
   lbl/.style={font=\scriptsize\bfseries},
   sub/.style={font=\scriptsize},
   num/.style={font=\tiny,gray!70},
   tg/.style={font=\scriptsize,blue!45!black},
   tgu/.style={font=\scriptsize,green!45!black}]
  \node[lbl,anchor=west] at (-0.55,3.25){(1)};
  \node[ti,anchor=west] at (-0.05,3.25){$Q$ and its ideal lattice $\mathcal{J}(Q)$};
  \node[pt,label=below:{\scriptsize$a$}] (qa) at (0.0,0.55){};
  \node[pt,label=below:{\scriptsize$b$}] (qb) at (0.8,0.55){};
  \node[pt,label=above:{\scriptsize$c$}] (qc) at (0.4,1.5){};
  \draw[gv] (qa)--(qc); \draw[gv] (qb)--(qc);
  \node[ti] at (0.4,-0.35){$Q$};
  \node[sub] at (1.2,1.0){$\rightsquigarrow$};
  \begin{scope}[shift={(2.95,0)}]
    \node[nd,label={below:\scriptsize$1{:}\varnothing$}]  (e)   at (0,0){};
    \node[nd,label={left:\scriptsize$2{:}\{a\}$}]         (A)   at (-0.65,0.85){};
    \node[nd,label={right:\scriptsize$3{:}\{b\}$}]        (B)   at (0.65,0.85){};
    \node[nd,label={right:\scriptsize$4{:}\{a,b\}$}]      (AB)  at (0,1.7){};
    \node[nd,label={above:\scriptsize$5{:}\{a,b,c\}$}]    (ABC) at (0,2.55){};
    \draw[le](e)--(A); \draw[le](e)--(B); \draw[le](A)--(AB); \draw[le](B)--(AB); \draw[le](AB)--(ABC);
    \node[ti,align=center] at (0,-0.9){$1$--$5$: a linear extension};
  \end{scope}
  \begin{scope}[shift={(7.0,0)}]
    \node[lbl,anchor=west] at (-1.0,3.25){(2)};
    \node[ti,anchor=west] at (-0.55,3.25){$c_{\mathrm{sub}}=\textnormal{\textsc{DownZeta}}(\mathbf 1)$};
    \node[halo] at (0,0){}; \node[halo] at (-0.65,0.85){}; \node[halo] at (0.65,0.85){}; \node[halo] at (0,1.7){};
    \node[nd,label={below:\scriptsize$\varnothing$}] (e2) at (0,0){};
    \node[nd,label={left:\scriptsize$\{a\}$}]        (A2) at (-0.65,0.85){};
    \node[nd,label={right:\scriptsize$\{b\}$}]       (B2) at (0.65,0.85){};
    \node[nd,label={right:\scriptsize$J{=}\{a,b\}$}] (AB2) at (0,1.7){};
    \node[nd,label={above:\scriptsize$\{a,b,c\}$}]   (ABC2) at (0,2.55){};
    \draw[le](e2)--(A2); \draw[le](e2)--(B2); \draw[le](A2)--(AB2); \draw[le](B2)--(AB2); \draw[le](AB2)--(ABC2);
    \node[tg,align=center] at (0.1,-0.8){the $4$ shaded ideals,\\ $c_{\mathrm{sub}}(\{a,b\})=4$};
  \end{scope}
  \begin{scope}[shift={(11.0,0)}]
    \node[lbl,anchor=west] at (-1.0,3.25){(3)};
    \node[ti,anchor=west] at (-0.55,3.25){$c_{\sup}=\textnormal{\textsc{UpZeta}}(\mathbf 1)$};
    \node[halou] at (-0.65,0.85){}; \node[halou] at (0,1.7){}; \node[halou] at (0,2.55){};
    \node[nd,label={below:\scriptsize$\varnothing$}] (e3) at (0,0){};
    \node[nd,label={left:\scriptsize$D{=}\{a\}$}]    (A3) at (-0.65,0.85){};
    \node[nd,label={right:\scriptsize$\{b\}$}]       (B3) at (0.65,0.85){};
    \node[nd,label={right:\scriptsize$\{a,b\}$}]     (AB3) at (0,1.7){};
    \node[nd,label={above:\scriptsize$\{a,b,c\}$}]   (ABC3) at (0,2.55){};
    \draw[le](e3)--(A3); \draw[le](e3)--(B3); \draw[le](A3)--(AB3); \draw[le](B3)--(AB3); \draw[le](AB3)--(ABC3);
    \node[tgu,align=center] at (0.1,-0.8){the $3$ shaded ideals,\\ $c_{\sup}(\{a\})=3$};
  \end{scope}
\end{tikzpicture}
\caption{The ideal lattice and the two zeta sweeps, for the three-point parent $Q$ with $a<c$ and $b<c$
(an arrow $x\to y$ means $x<y$). \textbf{(1)} The order ideals (down-sets) of $Q$ are
$\varnothing,\,\{a\},\,\{b\},\,\{a,b\},\,\{a,b,c\}$, each a subset of $Q$; ordered by inclusion they form the lattice $\mathcal{J}(Q)$, whose
arrows denote containment ($\subseteq$) and run from each ideal up to those that contain it. Its $\dd=5$ nodes
are processed in a \emph{linear extension} (the labels $1{:}\,\cdot$ through $5{:}\,\cdot$, each ideal listed
after every ideal it contains). Panels (2) and (3) redraw this same lattice; in them $J$ and $D$ are single
ideals, individual nodes of $\mathcal{J}(Q)$. \textbf{(2)}
$c_{\mathrm{sub}}(J)=\textsc{DownZeta}(\mathbf 1)(J)$ is the number of ideals contained in $J$: with $J=\{a,b\}$
these are the four shaded nodes $\varnothing,\{a\},\{b\},\{a,b\}$ (the down-cone of $J$), so $c_{\mathrm{sub}}(\{a,b\})=4$.
\textbf{(3)} $c_{\sup}(D)=\textsc{UpZeta}(\mathbf 1)(D)$ is the number of ideals containing $D$: with $D=\{a\}$
these are the three shaded nodes $\{a\},\{a,b\},\{a,b,c\}$ (the up-cone of $D$), so $c_{\sup}(\{a\})=3$. Each of $c_{\mathrm{sub}}$, $c_{\sup}$, and the higher transforms $M_j$ is filled in for
\emph{all} ideals at once by a single pass along the linear extension; this is what lets the kernel read a
parent's entire moment off $\mathcal{J}(Q)$ in one sweep.}
\label{fig:idealzeta}
\end{figure}
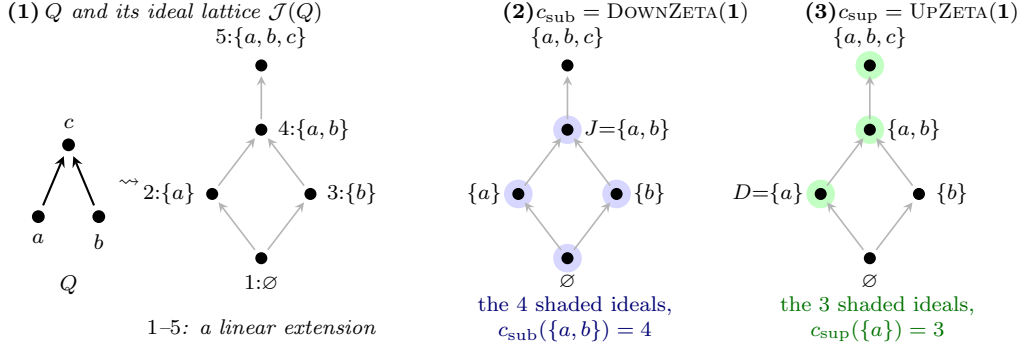

\emph{(i) Insertion is governed by the ideal lattice.} As in Section~\ref{sec:harvest}, an insertion is a
pair of ideals $(J,D)$, with $J=Q\setminus U$ encoding the up-set and $D$ the down-set, and the child's
antichain count is given by the insertion identity~\eqref{eq:childcount}, $\dd=c_{\mathrm{sub}}(J)+c_{\sup}(D)$,
where $c_{\mathrm{sub}}(J)$ and $c_{\sup}(D)$ count the order ideals contained in $J$ and containing $D$ in the
distributive ideal lattice $\mathcal{J}(Q)$~\cite{Birkhoff,Stanley}; equivalently
$c_{\mathrm{sub}}=\textsc{DownZeta}(\mathbf 1)$ and $c_{\sup}=\textsc{UpZeta}(\mathbf 1)$ are the down- and
up-cumulative sums of the constant function $\mathbf 1$ (summing $\mathbf 1$ over a set counts it). Hence
the contribution of $Q$'s insertions to $G(m{+}1,k)$ is a sum over such pairs,
$\sum_{J}\sum_{D}\bigl(c_{\mathrm{sub}}(J)+c_{\sup}(D)\bigr)^k$.

\emph{(ii) The binomial theorem decouples the insertion variable.} In the inner sum the first summand
$c_{\mathrm{sub}}(J)$ is fixed: it depends on the ideal $J$ alone, not on the insertion $D$, while only
$c_{\sup}(D)$ varies with $D$. Expanding $\bigl(c_{\mathrm{sub}}(J)+c_{\sup}(D)\bigr)^k$ by the binomial theorem
therefore separates each term into a $J$-only factor $\binom{k}{j}\,c_{\mathrm{sub}}(J)^{k-j}$, which pulls
outside the sum over $D$, and a $D$-only factor $c_{\sup}(D)^j$, which stays inside; summing the latter over
the admissible $D$ leaves exactly the moment $M_j$. The $\dd$-fold insertion-sum thus collapses to a fixed
combination of moments of $c_{\sup}$,
\begin{equation}\label{eq:collapse}
\sum_{D\subseteq B^\ast(J)}\bigl(c_{\mathrm{sub}}(J)+c_{\sup}(D)\bigr)^k
=\sum_{j=0}^{k}\binom{k}{j}\,c_{\mathrm{sub}}(J)^{\,k-j}\,M_j\!\bigl(B^\ast(J)\bigr)
\end{equation}
where $B^\ast(J)$ is the admissible region~\eqref{eq:bstar} (computed from one stored covering arc per
ideal) and $M_j$ is the $j$-th \emph{moment} of $c_{\sup}$ over the sub-ideals,
\[
M_j(X)\;=\;\textsc{DownZeta}\!\bigl(c_{\sup}^{\,j}\bigr)(X)\;=\;\sum_{D\subseteq X}c_{\sup}(D)^{\,j}.
\]
The argument of $\textsc{DownZeta}$ is a \emph{function} on the lattice, and $c_{\sup}^{\,j}$ is the $j$-th
\emph{pointwise} power of $c_{\sup}$ (its value at each ideal raised to the $j$). These moments therefore run
from the plain count $M_0=\textsc{DownZeta}(\mathbf 1)=c_{\mathrm{sub}}$ (the $0$th power $c_{\sup}^{\,0}$ is
the constant function $\mathbf 1$), through $M_1=\textsc{DownZeta}(c_{\sup})$ and
$M_2=\textsc{DownZeta}(c_{\sup}^{\,2})$, up to $M_K$.

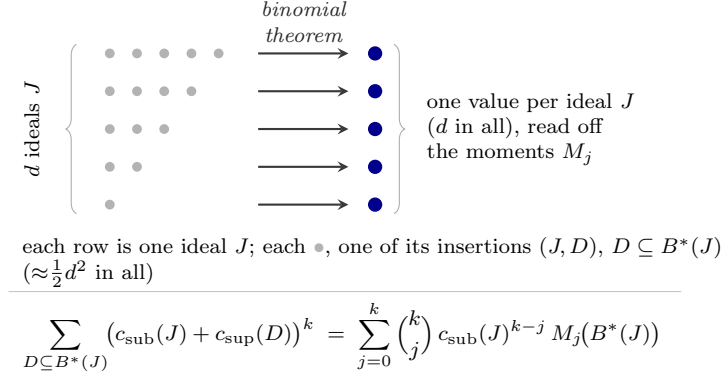
\begin{figure}[tbp]
\centering
\begin{tikzpicture}[x=1cm,y=1cm,>=stealth,
   din/.style={circle,fill=gray!60,inner sep=1.3pt},
   dout/.style={circle,fill=blue!55!black,inner sep=1.9pt},
   ti/.style={font=\scriptsize\itshape},
   sub/.style={font=\scriptsize},
   note/.style={font=\scriptsize\itshape,text=gray!30!black},
   ar/.style={->,thick,gray!45!black,shorten >=1.5pt,shorten <=1.5pt}]
  \foreach \r/\len in {0/5,1/4,2/3,3/2,4/1}{%
    \foreach \c in {1,...,\len}{\node[din] at (0.78+\c*0.36,2.0-\r*0.5){};}}
  \draw[decorate,decoration={brace,amplitude=4pt},gray!65] (0.7,-0.12)--(0.7,2.12);
  \node[sub,rotate=90,anchor=south] at (0.34,1.0){$\dd$ ideals $J$};
  \node[sub,anchor=north west,align=left] at (-0.15,-0.3){each row is one ideal $J$; each \textcolor{gray!60}{$\bullet$}, one of its insertions $(J,D)$, $D\subseteq B^\ast(J)$\\ (${\approx}\tfrac12\dd^2$ in all)};
  \foreach \r in {0,1,2,3,4}{\draw[ar] (3.05,2.0-\r*0.5)--(4.35,2.0-\r*0.5);}
  \node[note,align=center] at (3.7,2.42){binomial\\ theorem};
  \foreach \r in {0,1,2,3,4}{\node[dout] at (4.65,2.0-\r*0.5){};}
  \draw[decorate,decoration={brace,amplitude=4pt},gray!65] (4.9,2.12)--(4.9,-0.12);
  \node[sub,anchor=west,align=left] at (5.2,1.0){one value per ideal $J$\\ ($\dd$ in all), read off\\ the moments $M_j$};
  \draw[gray!45] (-0.2,-1.15)--(8.7,-1.15);
  \node[sub,anchor=west] at (-0.15,-1.75){$\displaystyle\sum_{D\subseteq B^\ast(J)}\!\bigl(c_{\mathrm{sub}}(J)+c_{\sup}(D)\bigr)^k\;=\;\sum_{j=0}^{k}\binom{k}{j}\,c_{\mathrm{sub}}(J)^{k-j}\,M_j\!\bigl(B^\ast(J)\bigr)$};
\end{tikzpicture}
\caption{The binomial collapse, the step that makes the kernel $O(\dd)$ per parent. For a fixed ideal $J$,
the inner sum over the admissible insertions $D\subseteq B^\ast(J)$ (about $\dd$ of them) would cost $O(\dd)$
term by term; expanding the $k$-th power by the binomial theorem~\eqref{eq:collapse} replaces it with a fixed
combination of the moments $M_j(B^\ast(J))$, each a single value read off the precomputed zeta transforms of
Figure~\ref{fig:idealzeta}. Summed over the $\dd$ ideals $J$ this yields $S_k(Q)$, so the
$\approx\tfrac12\dd^2$ admissible insertions are evaluated in $O(\dd)$ arithmetic rather than $O(\dd^2)$, and
no child $Q+z$ is ever formed.}
\label{fig:collapse}
\end{figure}

\emph{(iii) Each transform is one fast-zeta pass.} On the distributive lattice $\mathcal{J}(Q)$, whose
join-irreducibles are exactly the $|Q|$ principal ideals, both $c_{\sup}$ and every $M_j$ are computable
in $O(\dd\cdot n)$ by the fast-zeta circuits of Bj\"orklund \emph{et al.}~\cite{BHKKNP} (one sweep along a linear
extension; Figure~\ref{fig:sweep}).

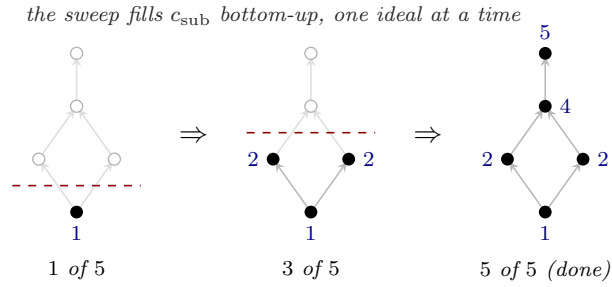
\begin{figure}[tbp]
\centering
\begin{tikzpicture}[x=1cm,y=1cm,>=stealth,
   fn/.style={circle,fill=black,inner sep=1.7pt},
   pn/.style={circle,draw=gray!60,fill=white,inner sep=1.6pt,line width=0.5pt},
   le/.style={semithick,gray!55,->,shorten >=1.3pt,shorten <=1.3pt},
   lep/.style={semithick,gray!25,->,shorten >=1.3pt,shorten <=1.3pt},
   val/.style={font=\scriptsize,blue!50!black},
   ti/.style={font=\scriptsize\itshape},
   wf/.style={dashed,red!55!black,line width=0.6pt}]
  \node[ti,text=gray!30!black] at (2.6,2.6){the sweep fills $c_{\mathrm{sub}}$ bottom-up, one ideal at a time};
  \begin{scope}
    \draw[lep](0,0)--(-0.5,0.7); \draw[lep](0,0)--(0.5,0.7); \draw[lep](-0.5,0.7)--(0,1.4);
    \draw[lep](0.5,0.7)--(0,1.4); \draw[lep](0,1.4)--(0,2.1);
    \node[fn] (e) at (0,0){}; \node[val,below=-1pt] at (e.south){$1$};
    \node[pn] at (-0.5,0.7){}; \node[pn] at (0.5,0.7){}; \node[pn] at (0,1.4){}; \node[pn] at (0,2.1){};
    \draw[wf] (-0.85,0.35)--(0.85,0.35);
    \node[ti,anchor=north] at (0,-0.5){$1$ of $5$};
  \end{scope}
  \node at (1.55,1.0){$\Rightarrow$};
  \begin{scope}[shift={(3.1,0)}]
    \draw[le](0,0)--(-0.5,0.7); \draw[le](0,0)--(0.5,0.7); \draw[lep](-0.5,0.7)--(0,1.4);
    \draw[lep](0.5,0.7)--(0,1.4); \draw[lep](0,1.4)--(0,2.1);
    \node[fn] (e2) at (0,0){}; \node[val,below=-1pt] at (e2.south){$1$};
    \node[fn] (a2) at (-0.5,0.7){}; \node[val,left=-1pt] at (a2.west){$2$};
    \node[fn] (b2) at (0.5,0.7){}; \node[val,right=-1pt] at (b2.east){$2$};
    \node[pn] at (0,1.4){}; \node[pn] at (0,2.1){};
    \draw[wf] (-0.85,1.05)--(0.85,1.05);
    \node[ti,anchor=north] at (0,-0.5){$3$ of $5$};
  \end{scope}
  \node at (4.65,1.0){$\Rightarrow$};
  \begin{scope}[shift={(6.2,0)}]
    \draw[le](0,0)--(-0.5,0.7); \draw[le](0,0)--(0.5,0.7); \draw[le](-0.5,0.7)--(0,1.4);
    \draw[le](0.5,0.7)--(0,1.4); \draw[le](0,1.4)--(0,2.1);
    \node[fn] (e3) at (0,0){}; \node[val,below=-1pt] at (e3.south){$1$};
    \node[fn] (a3) at (-0.5,0.7){}; \node[val,left=-1pt] at (a3.west){$2$};
    \node[fn] (b3) at (0.5,0.7){}; \node[val,right=-1pt] at (b3.east){$2$};
    \node[fn] (ab3) at (0,1.4){}; \node[val,right=-1pt] at (ab3.east){$4$};
    \node[fn] (abc3) at (0,2.1){}; \node[val,above=-1pt] at (abc3.north){$5$};
    \node[ti,anchor=north] at (0,-0.5){$5$ of $5$ (done)};
  \end{scope}
\end{tikzpicture}
\caption{The fast zeta as a single sweep, shown in three stages for $c_{\mathrm{sub}}=\textsc{DownZeta}(\mathbf
1)$ on the lattice $\mathcal{J}(Q)$ of Figure~\ref{fig:idealzeta}: a filled node already carries its value
$c_{\mathrm{sub}}$ (blue, the number of ideals contained in it), a hollow node is not yet computed, and the
dashed line marks the sweep frontier. Sweeping the ideals bottom-up in a linear extension, every sub-ideal that a
node depends on has already been visited, so all $\dd$ values are produced in one organized pass of cost
$O(\dd\,n)$ ($n=|Q|$ join-irreducibles, the fast-zeta transform of Bj\"orklund \emph{et al.}~\cite{BHKKNP}), rather than
the $O(\dd^2)$ of counting each node's sub-ideals separately. The higher moments $M_1,\dots,M_K$ and $c_{\sup}$
are each filled by the same kind of single pass.}
\label{fig:sweep}
\end{figure}

Composing these proves Proposition~\ref{prop:kernel}: the $K{+}1$ transforms cost $O(|\mathcal{J}(Q)|\,m\,K)$ in
total, and the combine step adds $O(|\mathcal{J}(Q)|\,K)$ per moment, so for fixed $K$ the per-parent contribution
costs $O(|\mathcal{J}(Q)|\,m)$, the same order as computing the parent's own antichain count rather than evaluating
each insertion separately.
\end{proof}

Advancing one diagonal is thus, per parent, free up to the zeta transform. Over the full sweep this gave
a $2.40\times$ speedup over an $O(\dd^2)$ reference, the gain concentrated on the heavy parents whose ideal
lattices are largest. This kernel is the methodological contribution of the present work. It transfers the entire
$(m{+}1)$-point moment off each $m$-point parent's ideal lattice in a single pass, by combining the
insertion identity~\eqref{eq:childcount} with the binomial collapse~\eqref{eq:collapse} and one zeta
transform over $\mathcal{J}(Q)$, generating and canonicalizing no $(m{+}1)$-point poset. To our knowledge this
single-pass moment transfer does not appear in the earlier enumeration pipelines, which record a parent's
own antichain count and advance the diagonal analytically~\cite{ErneStege,HeitzigReinhold,BrinkmannMcKay}
rather than reading its children's moment directly off the parent.

Algorithm~\ref{alg:kernel} states the procedure, and Figure~\ref{fig:pipeline} its dataflow. The inner double loop is the binomial
collapse~\eqref{eq:collapse}; everything else is a constant number of zeta passes over $\mathcal{J}(Q)$.

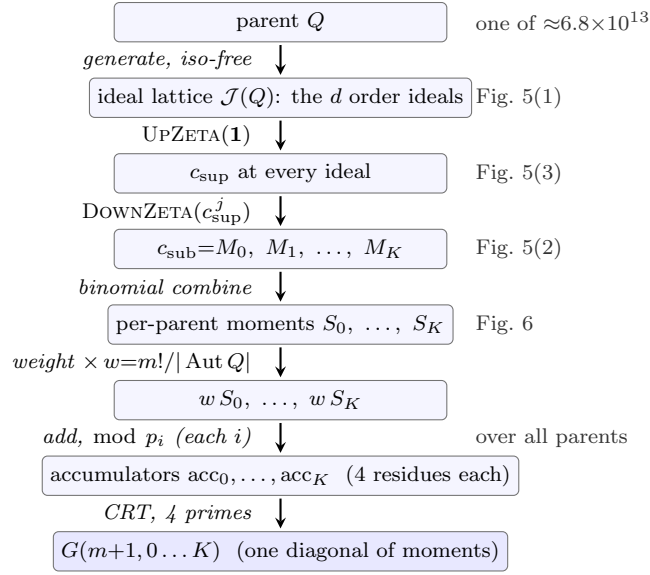
\begin{figure}[t]
\centering
\begin{tikzpicture}[x=1cm,y=1cm,>=stealth,
   bx/.style={draw=black!55,rounded corners=2pt,fill=blue!4,inner sep=3pt,align=center,
              font=\scriptsize,minimum width=4.4cm,minimum height=0.52cm},
   op/.style={font=\scriptsize\itshape,anchor=east},
   fg/.style={font=\scriptsize,gray!40!black,anchor=west},
   ar/.style={->,thick,shorten >=1pt,shorten <=1pt}]
  \node[bx] (b1) at (0,3.5){parent $Q$};
  \node[bx] (b2) at (0,2.5){ideal lattice $\mathcal{J}(Q)$: the $\dd$ order ideals};
  \node[bx] (b3) at (0,1.5){$c_{\sup}$ at every ideal};
  \node[bx] (b4) at (0,0.5){$c_{\mathrm{sub}}{=}M_0,\ M_1,\ \dots,\ M_K$};
  \node[bx] (b5) at (0,-0.5){per-parent moments $S_0,\ \dots,\ S_K$};
  \node[bx] (b6) at (0,-1.5){$w\,S_0,\ \dots,\ w\,S_K$};
  \node[bx] (b7) at (0,-2.5){accumulators $\mathrm{acc}_0,\dots,\mathrm{acc}_K$ \ (4 residues each)};
  \node[bx,fill=blue!9] (b8) at (0,-3.5){$G(m{+}1,0\dots K)$ \ (one diagonal of moments)};
  \draw[ar](b1)--(b2); \draw[ar](b2)--(b3); \draw[ar](b3)--(b4); \draw[ar](b4)--(b5);
  \draw[ar](b5)--(b6); \draw[ar](b6)--(b7); \draw[ar](b7)--(b8);
  \node[op] at (-0.25,3.0){generate, iso-free};
  \node[op] at (-0.25,2.0){$\textnormal{\textsc{UpZeta}}(\mathbf 1)$};
  \node[op] at (-0.25,1.0){$\textnormal{\textsc{DownZeta}}(c_{\sup}^{\,j})$};
  \node[op] at (-0.25,0.0){binomial combine};
  \node[op] at (-0.25,-1.0){weight $\times\,w{=}m!/|\Aut Q|$};
  \node[op] at (-0.25,-2.0){add, $\bmod\ p_i$ (each $i$)};
  \node[op] at (-0.25,-3.0){CRT, 4 primes};
  \node[fg] at (2.45,3.5){one of ${\approx}6.8{\times}10^{13}$};
  \node[fg] at (2.45,2.5){Fig.~\ref{fig:idealzeta}(1)};
  \node[fg] at (2.45,1.5){Fig.~\ref{fig:idealzeta}(3)};
  \node[fg] at (2.45,0.5){Fig.~\ref{fig:idealzeta}(2)};
  \node[fg] at (2.45,-0.5){Fig.~\ref{fig:collapse}};
  \node[fg] at (2.45,-2.0){over all parents};
\end{tikzpicture}
\caption{The kernel as a pipeline, one column per parent. For each unlabeled parent $Q$ (generated once,
isomorphism-free), the ideal lattice $\mathcal{J}(Q)$ is built and swept: one \textsc{UpZeta} pass gives
$c_{\sup}$, then $K{+}1$ \textsc{DownZeta} passes give $M_0{=}c_{\mathrm{sub}},\dots,M_K$
(Figure~\ref{fig:idealzeta}); the binomial combine (Figure~\ref{fig:collapse}) turns these into the
per-parent moments $S_0,\dots,S_K$. Each is weighted by $w=m!/|\Aut Q|$ and added to the running
accumulators modulo four $61$-bit primes; after all ${\approx}6.8\times10^{13}$ parents, the Chinese
Remainder Theorem reconstructs the exact moments $G(m{+}1,0\dots K)$. No $(m{+}1)$-point poset is ever
formed.}
\label{fig:pipeline}
\end{figure}

\begin{algorithm}[t]
\caption{Moment harvest of $G(m{+}1,0\ldots K)$ from the unlabeled $m$-point sweep.}
\label{alg:kernel}
\vspace{2pt}\hrule\vspace{5pt}
\newcommand{\cmt}[1]{\hfill{\small$\triangleright$ #1}}
\noindent
$\mathrm{acc}_k \gets 0$ \ for $k=0,\dots,K$\\[1pt]
\textbf{for} each unlabeled poset $Q$ on $[m]$ (generated once)\,\textbf{:}\\
\hspace*{1.5em}$L \gets$ the order ideals of $Q$, listed along a linear extension \cmt{the lattice $\mathcal{J}(Q)$; $|L|=\dd$}\\
\hspace*{1.5em}$w \gets m!\,/\,|\Aut Q|$ \cmt{automorphism-weighted multiplicity of $Q$}\\
\hspace*{1.5em}\textbf{for} each $J\in L$\,\textbf{:}\quad $B^\ast(J) \gets$ admissible region of $J$ \cmt{Eq.~\eqref{eq:bstar}, one covering arc per ideal}\\
\hspace*{1.5em}$c_{\sup} \gets \textsc{UpZeta}(\mathbf 1)$ on $L$ \cmt{$c_{\sup}(D)=\#\{I\in L: I\supseteq D\}$, the ideals above $D$}\\
\hspace*{1.5em}\textbf{for} $j=0,\dots,K$\,\textbf{:}\quad $M_j \gets \textsc{DownZeta}\!\bigl(c_{\sup}^{\,j}\bigr)$ on $L$ \cmt{$M_j(X)=\sum_{D\subseteq X}c_{\sup}(D)^{\,j}$; each $O(\dd\,n)$}\\
\hspace*{1.5em}$c_{\mathrm{sub}} \gets M_0$ \cmt{$c_{\mathrm{sub}}(J)=\#\{I\in L: I\subseteq J\}$, the ideals below $J$}\\
\hspace*{1.5em}\textbf{for} $k=0,\dots,K$\,\textbf{:}\\
\hspace*{3em}$S_k \gets \displaystyle\sum_{J\in L}\ \sum_{j=0}^{k}\binom{k}{j}\,c_{\mathrm{sub}}(J)^{\,k-j}\,M_j\!\bigl(B^\ast(J)\bigr)$ \cmt{binomial collapse, Eq.~\eqref{eq:collapse}}\\[2pt]
\hspace*{3em}$\mathrm{acc}_k \gets \mathrm{acc}_k + w\cdot S_k \pmod{p_i}$ \ for each $i=1,\dots,4$\\[1pt]
\textbf{return} $\mathrm{acc}_k = G(m{+}1,k)$ \ for $k=0,\dots,K$
\vspace{5pt}\hrule\vspace{4pt}
\noindent{\small\textbf{where} $\textsc{UpZeta}(f)[x]=\textstyle\sum_{y\in L:\,y\supseteq x}f[y]$ and
$\textsc{DownZeta}(f)[x]=\textstyle\sum_{y\in L:\,y\subseteq x}f[y]$ are the up- and down-cumulative sums of
$f$ over the ideal lattice $L=\mathcal{J}(Q)$, each computed in one linear-extension pass~\cite{BHKKNP}. Here
$\mathbf 1$ is the constant function~$1$, $c_{\sup}^{\,j}$ its $j$-th pointwise power, and $n=|Q|=m$ the
number of join-irreducibles of $L$.}
\vspace{3pt}\hrule
\end{algorithm}

\subsection{Arithmetic and reproducibility}
The moments are large integers: $G(16,4)$ has $40$ digits, and the accumulators sum $w\cdot S_k$ over all
$6.83\times10^{13}$ parents, each weight $w=15!/|\Aut Q|$ being itself as large as $15!\approx1.3\times10^{12}$.
Carrying these running sums in exact multiple-precision arithmetic would turn every one of those $10^{13}$
updates into a variable-length big-integer operation: each addition and multiplication would span several
machine words and need carry propagation and dynamic storage rather than a single instruction, and as this is
the innermost work of the sweep, that cost would dominate the whole computation. We avoid it by accumulating
modulo a fixed set of primes.
Reduction modulo a prime is a ring homomorphism, so a moment may be summed prime by prime, each partial sum
reduced as soon as it is formed, and its true value reconstructed only at the end; every operation in the hot
loop is then a fixed-width machine-word multiply and add, and the running accumulators never grow.

We accumulate modulo the four $61$-bit primes
\[
2305843009213693951,\quad 2305843009213693921,\quad 2305843009213693907,\quad 2305843009213693723
\]
the first of which is the Mersenne prime $2^{61}-1$. The width is chosen so that the product of two residues,
each below $2^{61}$, stays below $2^{122}$ and fits in a single $128$-bit word: each modular multiplication is
one machine multiply-and-reduce, with no overflow handling and no big-integer library in the inner loop.

The true moments are reconstructed by the Chinese Remainder Theorem, which recovers an integer uniquely from
its residues whenever the integer is smaller than the product of the moduli. The four primes multiply to a
value between $2^{243}$ and $2^{244}$, and even any three of them exceed $2^{182}$, far above the moments
themselves (the largest, $G(16,4)$, is below $2^{132}$, and $G(16,3)$ is about $2^{122}$). Any three of the
primes therefore already determine each moment exactly, and the fourth is redundant by design: it serves as an
arithmetic check (Figure~\ref{fig:crt}). Because the result is over-determined, a single corrupted residue, of the kind a transient
hardware fault could introduce over a run of $10^{13}$ parents, leaves the four residues mutually inconsistent
and throws the reconstruction far outside the known magnitude, so the fault is caught rather than silently
absorbed.

\begin{figure}[t]
\centering
\begin{tikzpicture}[x=1cm,y=1cm,>=stealth,
   bx/.style={draw=black!55,rounded corners=2pt,fill=blue!4,inner sep=3pt,align=center,font=\scriptsize},
   rb/.style={draw=black!50,rounded corners=1.5pt,fill=orange!9,inner sep=2.5pt,align=center,
              font=\scriptsize,minimum width=2.3cm},
   op/.style={font=\scriptsize\itshape},
   note/.style={font=\scriptsize,align=center},
   ar/.style={->,thick,shorten >=1pt,shorten <=1pt}]
  \node[bx,minimum width=2.5cm] (par) at (0,0){every parent $Q$\\ adds $w\,S_k$};
  \node[rb] (r1) at (3.85,1.05){$\mathrm{acc}_k \bmod p_1$};
  \node[rb] (r2) at (3.85,0.35){$\mathrm{acc}_k \bmod p_2$};
  \node[rb] (r3) at (3.85,-0.35){$\mathrm{acc}_k \bmod p_3$};
  \node[rb] (r4) at (3.85,-1.05){$\mathrm{acc}_k \bmod p_4$};
  \draw[thick] (par.east) -- (2.05,0);
  \draw[gray!60] (2.05,-1.05) -- (2.05,1.05);
  \draw[ar](2.05,1.05)--(r1.west); \draw[ar](2.05,0.35)--(r2.west);
  \draw[ar](2.05,-0.35)--(r3.west); \draw[ar](2.05,-1.05)--(r4.west);
  \node[op,anchor=north] at (1.7,-0.16){$\bmod\,p_i$};
  \draw[decorate,decoration={brace,amplitude=4pt},gray!65] (5.1,1.3)--(5.1,-1.3);
  \node[bx] (crt) at (6.55,0){CRT};
  \draw[ar](5.2,0)--(crt.west);
  \node[bx,fill=blue!9] (g) at (8.7,0){$G(m{+}1,k)$\\ exact, ${\approx}2^{132}$};
  \draw[ar](crt.east)--(g.west);
  \node[note,anchor=north,text width=11.6cm] at (3.9,-1.65){$p_1p_2p_3p_4>2^{243}$, and even any three exceed $2^{182}$, far above the moments ($G(16,4)<2^{132}$): three
    primes already reconstruct each moment exactly. The fourth over-determines the result, so a residue
    corrupted by a transient hardware fault makes the four inconsistent and throws the reconstruction outside
    the known magnitude, where it is caught.};
\end{tikzpicture}
\caption{Accumulation and reconstruction across parents. Each parent's weighted moment $w\,S_k$ is reduced
modulo the four $61$-bit primes and added to the running residues $\mathrm{acc}_k \bmod p_i$; after all
${\approx}6.8\times10^{13}$ parents the four residues are combined by the Chinese Remainder Theorem into the
exact moment $G(m{+}1,k)$. Three primes already suffice (their product exceeds the moment), and the fourth is
a redundancy check.}
\label{fig:crt}
\end{figure}

The sweep parallelizes
across disjoint sets of $15$-point parents. The moment $G(15,4)$ is obtained from the $15$-point posets directly (the same sweep, tallying
$\dd(Q)^4$ per poset), and the constant $A$ of~\eqref{eq:Adef} from the enumeration of
posets on at most $14$ points. Isomorphism handling uses \texttt{nauty}~\cite{nauty}.

\section{Computation and validation}\label{sec:validation}

\subsection{Running the sweep}\label{sec:sweep}
One representative of each unlabeled
$15$-point poset was emitted by the generator \texttt{genposetg} (the poset mode of the
\texttt{nauty}/\texttt{gtools} suite~\cite{nauty}), which realises the orderly generation of
Section~\ref{sec:method}.

The parallelism is over the $15$-point parents themselves. The generator accepts a parameter restricting
it to one block of a partition of its output, so assigning the blocks to separate workers divides the set
of parents into disjoint subsets that are swept independently. Two properties make this need no
communication between workers. First, orderly generation accepts or rejects each poset by a local
canonical-form test, so no worker has to see another's output to know which parents are its own. Second,
the kernel's contribution is additive across parents (each parent adds $w\cdot s$ to the accumulators of
Algorithm~\ref{alg:kernel}), so the partial sums formed on disjoint blocks combine by addition. Each worker
ran the kernel (the C program \texttt{poset\_moment\_filter}) over its
block of parents, maintaining the four residues $\mathrm{acc}_k \bmod p_i$ for $k=0,\dots,4$ and $i=1,\dots,4$.

The full moments are then recovered by summing the per-block partial residues prime by prime and applying
the Chinese Remainder Theorem as in Section~\ref{sec:method}; this harvest is a single streaming pass over the
worker outputs, with no intermediate posets retained and a single $128$-bit accumulator sufficing for the
inner products of Algorithm~\ref{alg:kernel}. The moment $G(15,4)$ was produced by sweeping the $15$-point posets directly, and the constant $A$ by
sweeping the posets on at most $14$ points; the insertion harvest above is reserved for the $16$-point
moment $G(16,3)$. The kernel and harvest code, the four primes, and a
per-shard residue table are listed under Data availability.

\subsection{Validation}
The sweep over all $A000112(15)=68{,}275{,}077{,}901{,}156$ unlabeled $15$-point posets yields, after CRT
reconstruction, the moments $G(16,0\ldots4)$. These satisfy the checks below; one of them is independent
of the computation itself: a residue predicted by the modular periodicity of $P(n)$.

\begin{itemize}
\item \textbf{Completeness.} The total parent count returned by the sweep equals
$A000112(15)=68{,}275{,}077{,}901{,}156$ exactly, which would fail were any parent omitted or counted
twice.
\item \textbf{Weight checksum.} $\sum_Q w \equiv P(15) \pmod{p}$ for each of the four primes, with
$P(15)=77{,}567{,}171{,}020{,}440{,}688{,}353{,}049{,}939$.
\item \textbf{Lower anchors.} $G(16,0)=P(16)$, and $G(16,1)$, $G(16,2)$ agree with their
Ern\'e--Stege values. At $k=0$ the harvest sum~\eqref{eq:harvest} reduces to the weighted count of all
admissible insertions, so the identity $G(16,0)=P(16)$ would fail if any insertion were dropped or
double-counted.
\item \textbf{Congruence.} A congruence recorded in the OEIS entry for \texttt{A001035}~\cite{OEIS}
predicts $P(19)\equiv 163{,}279{,}579 \pmod{232{,}792{,}560}$, the modulus being
$\operatorname{lcm}(1,\dots,20)=2^4\cdot3^2\cdot5\cdot7\cdot11\cdot13\cdot17\cdot19$; the broader
modular-finiteness context for such counting sequences is the Specker--Blatter
theorem~\cite{SpeckerBlatter,FischerMakowsky}. The computed value agrees,
$P(19)\equiv 163{,}279{,}579 \pmod{232{,}792{,}560}$.
\item \textbf{Magnitude.} $\log_2 P(19)\approx128.93$ continues the smooth progression of the known terms
($\log_2 P(n)\approx76.4,\,86.0,\,96.1,\,106.6,\,117.5$ for $n=14,\dots,18$); a gross error in $G(16,3)$ would
fall well outside this trend.
\end{itemize}

\section{Results}\label{sec:results}

Table~\ref{tab:results} collects the principal outputs. The headline value is the $39$-digit integer
\[
P(19)=646{,}099{,}441{,}937{,}791{,}106{,}493{,}755{,}218{,}560{,}442{,}089{,}979
\]
the first new term of \texttt{A001035} since Brinkmann and McKay reached $n=18$, and with it the
labeled topology count $T(19)=$\texttt{A000798}$(19)$ obtained from the Stirling transform~\eqref{eq:stirling}.
We are not aware of any independent recomputation of $P(17)$ or $P(18)$, nor of any prior value
for $P(19)$, so the result rests on the internal consistency established in Section~\ref{sec:validation}.
Among those checks the congruence is the most stringent external one: the residue
$P(19)\equiv163{,}279{,}579 \pmod{232{,}792{,}560}$ was recorded in the OEIS as a \emph{prediction} for the
nineteenth term, and the computed value reproduces it exactly.

The same sweep yields the new $16$-point moments $G(16,3)$ and $G(16,4)$. Of the two moments
in~\eqref{eq:p19} not absorbed into the constant~$A$, only $G(16,3)$ requires data beyond the catalogued
range of at most $15$ points: it is a $16$-point moment, whereas $G(15,4)$ is a moment over the catalogued $15$-point posets, summed directly from their
antichain-count histogram, and demands no frontier advance.
The Ern\'e--Stege reduction is structured so that this $16$-point moment cannot be folded into~$A$, unlike the
lower ones. The moment $G(16,4)$, produced by the same pass, is the $16$-point fourth-power moment that enters
the reduction~\eqref{eq:reduction} for $P(20)$; we bank it for that step and report the supporting moment
$G(15,4)$ as well.

The route combines the classical Ern\'e--Stege reduction with the isomorphism-free harvest of Heitzig and
Reinhold, carried one diagonal beyond the published record. Reaching that diagonal at this scale rests on the
moment kernel of Section~\ref{sec:kernel}, which transfers the entire $(m{+}1)$-point moment off each
$m$-point parent's ideal lattice in a single pass. Its novelty lies in reading a child's moment directly
off the parent, where the classical pipelines record only the parent's own count.
The principal contribution is the values themselves, together with the supporting moments that follow from the
same computation.

\begin{table}[tb]
\centering
\small
\begin{tabular}{@{}ll@{}}
\toprule
Quantity & Value \\
\midrule
$P(19)=$\texttt{A001035}$(19)$ & $646099441937791106493755218560442089979$ \\
$T(19)=$\texttt{A000798}$(19)$ & $689054943207246404281592791142107048261$ \\
$G(16,3)$ & $5322963351172775869497071016032650486$ \\
$G(16,4)$ & $2954997625790351969485154266039478036626$ \\
\midrule
$G(16,0)=P(16)$ & $83480529785490157813844256579$ \\
$G(16,1)$ & $28441643117705315333254490986318$ \\
$G(16,2)$ & $11344858065618251316427764256980898$ \\
\texttt{A000798}$(17)$ & $134137950093337880672321868725846$ \\
\bottomrule
\end{tabular}
\caption{Principal results. $P(19)$ has $39$ digits. $G(16,3)$ is the new moment required by the
reduction; $G(16,4)$, also apparently new, enters the $20$-point reduction. The lower block collects the
validation anchors of Section~\ref{sec:validation}: $G(16,0)=P(16)$, with $G(16,1)$ and $G(16,2)$ matching
their Ern\'e--Stege values, and the known topology count \texttt{A000798}$(17)$, which the Stirling
transform~\eqref{eq:stirling} reproduces as a check.}
\label{tab:results}
\end{table}

\section{Concluding remarks}

This places the labeled frontier at $n=19$, one step beyond the $P(18)$ of Brinkmann and
McKay~\cite{BrinkmannMcKay}. The same run also yields the moment $G(16,4)$, which enters the corresponding
reduction~\eqref{eq:reduction} for $P(20)$; the remaining inputs to that reduction include a moment at the
$17$-point level, whose evaluation would extend the same approach one diagonal further, over the larger
family of $16$-point posets.

\section*{Data availability}
The values $P(19)$, $T(19)$, and $G(16,k)$ for $k\le 4$, together with the kernel and harvest code, the
$O(\dd^2)$ reference implementation, the four CRT primes, a per-shard residue table
(one row per shard, summing prime by prime to the reconstructed $G(16,k)$), the antichain-count histograms of
the labeled posets on at most $15$ points, and the script that reconstructs the constant $A$ and the moment
$G(15,4)$ from those histograms and checks them against the published values, are available in the public
repository \url{https://github.com/Rafael-Ayala/posets-and-topologies-19}. The new terms
$P(19)=$\texttt{A001035}$(19)$ and $T(19)=$\texttt{A000798}$(19)$ have been submitted to the OEIS, together
with the moment sequences $m\mapsto G(m,k)=\sum_{Q}\dd(Q)^k$ (summed over the labeled posets $Q$ on $[m]$)
for each $k=1,\dots,4$ as new entries, the case $k=0$ being \texttt{A001035}. Since $\dd(Q)$ is the number
of order ideals of $Q$, the $k=1$ sequence counts the pairs $(Q,D)$ of a labeled poset and one of its
order ideals (down-sets).


\begin{thebibliography}{12}
\bibitem{BrinkmannMcKay} G.~Brinkmann and B.~D.~McKay, \emph{Posets on up to 16 points}, Order
\textbf{19} (2002), no.~2, 147--179.
\bibitem{ErneStege} M.~Ern\'e and K.~Stege, \emph{Counting finite posets and topologies}, Order
\textbf{8} (1991), 247--265.
\bibitem{HeitzigReinhold} J.~Heitzig and J.~Reinhold, \emph{The number of unlabeled orders on fourteen
elements}, Order \textbf{17} (2000), 333--341.
\bibitem{BHKKNP} A.~Bj\"orklund, T.~Husfeldt, P.~Kaski, M.~Koivisto, J.~Nederlof and P.~Parviainen,
\emph{Fast zeta transforms for lattices with few irreducibles}, in \emph{Proc.\ 23rd Annual ACM--SIAM
Symposium on Discrete Algorithms (SODA)}, 2012, pp.~1436--1444; expanded in ACM Trans.\ Algorithms
\textbf{12} (2016), no.~1, Art.~4.
\bibitem{nauty} B.~D.~McKay and A.~Piperno, \emph{Practical graph isomorphism, II}, J.~Symbolic Comput.
\textbf{60} (2014), 94--112.
\bibitem{McKay98} B.~D.~McKay, \emph{Isomorph-free exhaustive generation}, J.~Algorithms \textbf{26}
(1998), no.~2, 306--324.
\bibitem{KR} D.~J.~Kleitman and B.~L.~Rothschild, \emph{Asymptotic enumeration of partial orders on a
finite set}, Trans.\ Amer.\ Math.\ Soc.\ \textbf{205} (1975), 205--220.
\bibitem{SpeckerBlatter} C.~Blatter and E.~Specker, \emph{Recurrence relations for the number of labeled
structures on a finite set}, in \emph{Logic and Machines: Decision Problems and Complexity}, Lecture
Notes in Comput.\ Sci.\ \textbf{171}, Springer, 1984, pp.~43--61.
\bibitem{FischerMakowsky} E.~Fischer and J.~A.~Makowsky, \emph{Extensions and limits of the
Specker--Blatter theorem}, in \emph{Proc.\ 32nd EACSL Annual Conference on Computer Science Logic (CSL)},
LIPIcs \textbf{288}, 2024, Art.~26.
\bibitem{OEIS} OEIS Foundation Inc., \emph{The On-Line Encyclopedia of Integer Sequences},
\texttt{https://oeis.org}, sequences \texttt{A001035}, \texttt{A000798}, \texttt{A000112} (2026).
\bibitem{Birkhoff} G.~Birkhoff, \emph{Rings of sets}, Duke Math.\ J.\ \textbf{3} (1937), no.~3, 443--454.
\bibitem{Stanley} R.~P.~Stanley, \emph{Enumerative Combinatorics, Volume~1}, 2nd ed., Cambridge Studies in
Advanced Mathematics \textbf{49}, Cambridge University Press, 2012.
\end{thebibliography}
\end{document}